\documentclass[a4paper]{article}
\usepackage{amsmath}
\usepackage{amsfonts}
\usepackage{amssymb}
\usepackage{amsthm}
\usepackage{tikz}
\usepackage{tikz-cd}

\newcommand{\iso}{\ensuremath{\cong}}
\newcommand{\Z}[1][]{\ensuremath{\mathbb{Z}_{#1}}}
\newcommand{\Q}{\ensuremath{\mathbb{Q}}}

\newcommand{\N}{\ensuremath{\mathbb{N}}}

\newcommand{\proP}[2][p]{\ensuremath{\widehat{#2}_{(#1)}}}
\newcommand{\nsgp}[1][]{\ensuremath{\triangleleft_{\rm #1}}}

\newtheorem{theorem}{Theorem}
\newtheorem*{thmquote}{Theorem}
\newtheorem{prop}[theorem]{Proposition}

\newtheorem{clly}[theorem]{Corollary}
\theoremstyle{definition}
\newtheorem{defn}[theorem]{Definition}

\theoremstyle{remark}
\newtheorem*{rmk}{Remark}

\title{Profinite properties of RAAGs and special groups}
\author{Robert Kropholler and Gareth Wilkes}
\begin{document}
\maketitle
\begin{abstract}
In this paper we prove that RAAGs are distinguished from each other by their pro-$p$ completions for any choice of prime $p$, and that RACGs are distinguished from each other by their pro-2 completions. We also give a new proof that hyperbolic virtually special groups are good in the sense of Serre. Furthermore we give an example of a property of the underlying graph of a RAAG that translates to a property of the profinite completion.
\end{abstract}
Right-angled Artin groups (RAAGs) have been the subject of much recent interest, especially because of their rich subgroup structure; in particular every special group embeds in a RAAG \cite{haglund_special_2008}. Furthermore RAAGs are linear and have excellent residual properties. 
Here we will show that RAAGs, and the closely related right-angled Coxeter groups (RACGs), are in fact completely determined by their finite quotient groups. The proofs will rely principally on the cohomological rigidity result of Koberda \cite{koberda12}.

First let us recall some definitions.
\begin{defn}
Given a (finite simplicial) graph $\Gamma$, the {\em right-angled Artin group} $A(\Gamma)$ is the group with generating set $V(\Gamma)$ with the relation that vertices $v,w$ commute imposed whenever $v,w$ span an edge of $\Gamma$. The {\em right-angled Coxeter group} $C(\Gamma)$ is the quotient of $A(\Gamma)$ with the additional constraint that each generator has order 2.
\end{defn}
\begin{defn}
Given a discrete group $G$, the \emph{profinite completion} of $G$ is the inverse limit of the system of groups 
\[\hat G = \varprojlim_{N\nsgp[f] G} G/N\]
where $N$ ranges over the finite index normal subgroups of $G$. This is a compact Hausdorff topological group. Similarly one may define the pro-$p$ completion $\hat G_{(p)}$ as the inverse limit of all finite quotients of $G$ which are $p$-groups, for $p$ a prime.

The isomorphism type of a right-angled Artin group $A(\Gamma)$ uniquely determines the graph $\Gamma$ up to isomorphism; this fact was first established by Droms \cite{droms87}. We use a stronger cohomological criterion proved by Koberda \cite{koberda12}.

\begin{thmquote}[Koberda \cite{koberda12}]
Let $\Gamma,\Gamma'$ be finite graphs. Then $\Gamma\iso\Gamma'$ if and only if there is an isomorphism of cohomology groups
\[ H^\ast(A(\Gamma);\Q)\iso H^\ast(A(\Gamma');\Q)\]
in dimensions one and two, which respects the cup product.
\end{thmquote}

The proof relied solely on the following fact: each vertex $v\in\Gamma$ is dual to a cohomology class $f_v\in H^1(A(\Gamma);\Q)$ for which the map
\[ f_v\smallsmile\bullet: H^1(A(\Gamma);\Q)\to H^2(A(\Gamma);\Q)\]
has rank precisely the degree of the vertex $v$; moreover a vertex $w\in\Gamma$ is adjacent to $v$ precisely if $f_v\smallsmile f_w$ is non-zero. 

Now the class $f_v\smallsmile f_w$ is dual to an embedded 2-torus in the Salvetti complex of $A(\Gamma)$, hence gives a primitive element of $H^2(A(\Gamma);\Z)$. It follows that changing the coefficient field \Q{} to a finite field $\Z/p$ (for $p$ a prime) changes neither the rank of the above map, nor the adjacency condition following it. Hence Koberda's cohomological rigidity result also holds with coefficient field $\Z/p$. 

It remains to show that the pro-$p$ completion of our right-angled Artin group detects the cohomology in dimensions one and two, and the cup product. As we will discuss later, it is frequently the case for groups arising in low-dimensional topology that the cohomology of a group is determined by its profinite completion. In low dimensions however we always have substantial control. See \cite{serre13} for definitions and basic properties of profinite cohomology; the definitions largely parallel those for discrete groups. In particular there is a natural notion of cup product and the natural map from $G$ to its profinite completion induces a map on cohomology respecting the cup product.

\begin{prop}\label{lowdimgoodness}
Let $G$ be a discrete group and $p$ a prime.  
\begin{itemize}
\item $H^1(\proP{G};\Z/p)\to H^1(G;\Z/p)$ is an isomorphism;
\item $H^2(\proP{G};\Z/p)\to H^2(G;\Z/p)$ is injective; and
\item if $H^{1+1}$ denotes that part of second cohomology generated by cup products of elements of $H^1$, then $H^{1+1}(\proP{G};\Z/p)\to H^{1+1}(G;\Z/p)$ is an isomorphism;
\end{itemize}
where all the maps are the natural ones induced by $G\to \proP{G}$.
\end{prop}
\begin{proof}
The first point is a trivial consequence of the fact that $H^1(-;\Z/p)$ is naturally isomorphic to ${\rm Hom}(-;\Z/p)$ in either the category of discrete groups or the category of pro-$p$ groups. The third point follows from the first two and naturality of the cup product. The second is a special case of Exercise 2.6.1 of \cite{serre13}; we give here an explicit proof in dimension two in terms of extensions.

Recall that $H^2(-;\Z/p)$ classifies central extensions of $G$ by $\Z/p$ both for discrete and pro-$p$ groups (see Section 6.8 of \cite{RZ00} for the profinite theory). Take a central extension $H$ of \proP{G} by $\Z/p$ representing $\xi\in H^2(\proP{G};\Z/p)$. Then the pull-back 
\[P= \{(g,h)\in G\times H\text{ such that } \pi(h)=\iota(g)\}\]
(where $\pi:H\to \proP{G}$ and $\iota:G\to\proP{G}$ are the obvious maps) gives a central extension of $G$ by $\Z/p$ representing $\iota^\ast(\xi)$. If $\iota^\ast(\xi)=0$ in $H^2$ then the extension splits; so there is a group-theoretic section $s:G\to P$. Now $s$ induces a map $\hat{s}:\proP{G}\to \proP{P}$. Furthermore $H$ is a pro-$p$ group so that the projection ${\rm pr_2}:P\to H$ induces a map $\hat{\rm pr_2}:\proP{P}\to H$; then $\hat{\rm pr_2}\hat{s}$ is a section of $H\to \proP{G}$ and so $\xi$ was a trivial extension also.
\[\begin{tikzcd}
\Z/p\ar[hook]{rr}\ar{rd}\ar{dd} && P\ar[two heads]{rr}\ar{dd}[anchor=south, xshift=-8pt, yshift=10pt]{\rm pr_2}\ar{rd} && G\ar[bend left, dotted]{ll}{s} \ar{dd}[anchor=south, xshift=-2pt, yshift=10pt]{\iota}\ar{rd} & \\
& \Z/p\ar[crossing over]{rr}\ar{ld} && \proP{P}\ar[crossing over, two heads]{rr}\ar{ld}{\hat{\rm pr_2}} && \proP{G}\ar[bend left, dotted, crossing over]{ll}[anchor=east]{\hat s} \ar{ld}{\iso}\\
 \Z/p\ar[hook]{rr} && H \ar[two heads]{rr}{\pi} && \proP{G} &
\end{tikzcd}\]
\end{proof}
\begin{rmk}
Note that a diagram chase applied to the lower parallelogram in the above diagram shows that in fact $H\iso\proP{P}$. Thus the above analysis also illustrates why the map on $H^2$ may fail to be surjective; for any central extension of \proP{G} yielding a given extension $P$ of $G$ must be \proP{P}; however there is no {\it a priori} reason that the map from $\Z/p$ to \proP{P} need be injective. See the example at the end of \cite{lorensen10} for an example where the map on $H^2$ fails to be surjective.
\end{rmk}
Recall that for a RAAG, the dimension two cohomology is in fact generated by cup products; thus in dimensions one and two, the algebra $H^\ast(A(\Gamma);\Z/p)$ is determined by the pro-$p$ completion $\widehat{A(\Gamma)}_{(p)}$; hence we have proved
\begin{theorem}\label{RAAGsrigid}
Let $\Gamma,\Gamma'$ be finite graphs and $p$ a prime. Then $\widehat{A(\Gamma)}_{(p)}\iso\widehat{A(\Gamma')}_{(p)}$ if and only if $\Gamma\iso\Gamma'$.
\end{theorem}

In fact much more is true about the cohomology of the pro-$p$ completion of $A(\Gamma)$; in particular:
\begin{theorem}[Lorensen \cite{lorensen08}, \cite{lorensen10}]\label{goodRAAGs}
The map from a right-angled Artin group to its pro-$p$ completion (or profinite completion) induces an isomorphism of mod-$p$ cohomology for any prime $p$.
\end{theorem}

We can extend Theorem \ref{RAAGsrigid} to right-angled Coxeter groups by noting that there are natural isomorphisms
\[ H^1(C(\Gamma);\Z/2)\iso H^1(A(\Gamma);\Z/2) \]
and 
\[ H^2(C(\Gamma);\Z/2)\iso H^2(A(\Gamma);\Z/2) \oplus (\Z/2)^{|V(\Gamma)|}\]
where the second summand above derives from the relations $v^2=1$. The quotient map on $H^2$ which restricts to the isomorphism of the first summand with $H^2(A(\Gamma);\Z/2)$ is induced by the natural map $A(\Gamma)\to C(\Gamma)$. This quotient map is unique (i.e. does not depend on the presentation of $C(\Gamma)$ as a particular right-angled Coxeter group) in the following sense. Modulo 2, we have the relations $(a+b)^2 = a^2 + b^2$ so that the image of the squaring map $a\to a\smallsmile a$ is a subgroup $\Sigma$ of $H^2(C(\Gamma))$, the image of the diagonal subgroup of $(H^1(C(\Gamma)))^2$. The second summand $(\Z/2)^{|V(\Gamma)|}$ is precisely this subgroup $\Sigma$.

Thus the structure of the algebra $H^\ast(A(\Gamma))$ in dimensions one and two is determined by the behaviour of $H^\ast(C(\Gamma);\Z/2)$ in those dimensions, with the cup product map being given by the canonical map
\[(H^1(A(\Gamma)))^2\stackrel{\iso}{\to} (H^1(C(\Gamma)))^2 \stackrel{\smallsmile}{\to} H^2(C(\Gamma)) \to H^2(C(\Gamma))/\Sigma \iso H^2(A(\Gamma))\]
described above. Proposition \ref{lowdimgoodness} shows this algebra to be an invariant of the pro-2 completion, so that we have:
\begin{theorem}
Let $\Gamma,\Gamma'$ be finite graphs. Then $\widehat{C(\Gamma)}_{(2)}\iso\widehat{C(\Gamma')}_{(2)}$ if and only if $\Gamma\iso\Gamma'$.
\end{theorem}
\begin{rmk}
Proposition \ref{lowdimgoodness} was sufficient to prove the Theorem; in fact the map 
\[ H^2(\widehat{C(\Gamma)}_{(2)};\Z/2)\hookrightarrow H^2(C(\Gamma);\Z/2)\]
is surjective. It is in fact generated by cup products; the first summand of
\[ H^2(C(\Gamma);\Z/2)\iso H^2(A(\Gamma);\Z/2) \oplus (\Z/2)^{|V(\Gamma)|}\]
is generated by cup products of distinct elements on $H^1$, just as for $A(\Gamma)$. The second is generated by cup products $\alpha_w=w^\ast\smallsmile w^\ast$ for elements $w^\ast$ of $H^1$ dual to a generator $w$ of $C(\Gamma)$. Via the Hopf formula, these $\alpha_w$ correspond to the relations $w^2=1$. We may directly see that these $\alpha_w$ are in the image of 
\[ H^2(\widehat{C(\Gamma)}_{(2)};\Z/2)\hookrightarrow H^2(C(\Gamma);\Z/2)\]
as follows. Recall that the obstruction to surjectivity was that the pro-2 completion of certain short exact sequences may fail to be exact. Considering a classifying space for $C(\Gamma)$ one sees that $\alpha_w$ is the image of the non-zero element of $H^2(\Z/2;\Z/2)$ under the map to $H^2(C(\Gamma))$ induced by the map $w^\ast:C(\Gamma)\to \Z/2$. The extensions given by these cohomology classes are then pull-backs
\[\begin{tikzcd}
1\ar{r} & \Z/2 \ar{r} \ar{d} & P \ar{r}\ar{d} & C(\Gamma)\ar{r}\ar{d} & 1\\
1\ar{r} & \Z/2 \ar{r} & \Z/4 \ar{r} & \Z/2\ar{r}  &1
\end{tikzcd}\]
in which the map to $\Z/4$ witnesses the fact that $\Z/2\to\proP[2]{P}$ is an inclusion.

More explicitly, such an extension $P$ has presentation
\[\big< V(\Gamma)\,\big|\, [u,v]=1\text{ if }(u,v)\in E(\Gamma), v^2=1 \text{ for }v\neq w, w^4=1, w^2\in Z(P)\big>\]
where the map to $C(\Gamma)$ kills the central element $w^2$ and the map to $\Z/4$ kills all generators except $w$.
\end{rmk}
We made heavy use of the cohomology of the profinite completions of RAAGs and RACGs, so let us digress and study the following property. A group $G$ is {\em good} if the natural map on cohomology induced by $G \to \hat G$ is an isomorphism 
\[H^n(\hat G; M)\stackrel{\iso}{\to} H^n(G; M)\] 
for every finite $G$-module $M$ and every $n\geq 1$. 
\end{defn}
Note that this map $G\to \hat G$ for any group $G$ respects the cup product; for the cup product is defined for cohomology of profinite groups by precisely the same formulae as for abstract groups. Thus for a good group $G\to\hat G$ will not only induce an isomorphism of groups, but an isomorphism of graded algebras $H^\ast(\hat G; M)\iso H^\ast(G;M)$ under the cup product.

Goodness is preserved under taking extensions by good groups (see \cite{serre13}), and passing to finite index subgroups or overgroups. Recall:
\begin{thmquote}[Grunewald, Jaikin-Zapirain, Zalesskii \cite{GJZZ08}]
Finitely generated Fuchsian groups are good.
\end{thmquote}
Hence all virtually-fibred 3-manifold groups are good; due to work of Agol \cite{agol_virtual_2013} and Wise \cite{wise_structure_2011} this includes all hyperbolic 3-manifolds. 

Goodness is also preserved under taking amalgamated free products and HNN extensions, given suitable conditions. Recall that the {\em full profinite topology} on a group $H$ is the topology induced by the map $H\to\hat H$. An amalgamated free product $G=A\ast_C B$ is said to be {\em efficient} if $G$ is residually finite, if $A,B,C$ are closed in the profinite topology on $G$, and if $G$ induces the full profinite topology on $A,B$ and $C$. In particular if $G$ is LERF (and $A,B,C$ are finitely generated) this condition will be satisfied. Similarly one may define efficient HNN extensions. Then:
\begin{theorem}[Proposition 3.6 of \cite{GJZZ08}]\label{goodsplitting}
An efficient amalgamated free product or HNN extension of good groups is good.
\end{theorem}

We may now prove that all hyperbolic virtually special groups are good. This was first proved by Minasyan and Zalesskii \cite{MZ15} using virtual retraction properties; we give another proof using cube complex hierarchies.

\begin{defn}
Let $\mathcal{QVH}$ be the smallest class of hyperbolic groups closed under the following operations. 
\begin{enumerate}
\item The trivial group is in $\mathcal{QVH}$.
\item If $A, B, C\in\mathcal{QVH}$ and $C$ is quasiconvex in $G = A\ast_CB$, then $G\in\mathcal{QVH}$.
\item If $A, B\in\mathcal{QVH}$ and $B$ is quasiconvex in $G = A\ast_B$, then $G\in\mathcal{QVH}$.
\item If $H$ is commensurable to $G$ and $G\in\mathcal{QVH}$, then $H\in\mathcal{QVH}$.
\end{enumerate}
\end{defn}

We have the following useful characterisation of the groups in $\mathcal{QVH}$. 
\begin{thmquote}[Wise \cite{wise_structure_2011}]
A hyperbolic group is in $\mathcal{QVH}$ if and only if it is virtually special.
\end{thmquote}  
Furthermore recall
\begin{thmquote}[Haglund-Wise \cite{haglund_special_2008}]
Quasi-convex subgroups of hyperbolic virtually special groups are separable.
\end{thmquote}
With this in mind we can now prove. 

\begin{theorem}
Hyperbolic virtually special groups are good.
\end{theorem}
\begin{proof}
As noted above, we are free to pass to an arbitrary finite index subgroup of $G$ and prove goodness there. We define a measure of complexity for a special group $H$. Set $n(H)$ to be the minimal dimension of a CAT(0) cube complex $X$ on which $H$ acts with special quotient. After subdividing, a hyperplane in this complex is an embedded 2-sided cubical subcomplex, and $H$ splits as an HNN extension or amalgamated free product over the stabiliser of this hyperplane. Iterating this process yields a rooted tree of groups in which each vertex has either two or three descendants (depending on whether the vertex splits as an HNN extension or amalgamated free product). Because $H$ is in $\cal QVH$, this tree is finite and each branch terminates in the trivial group. Set $m(H)$ to be the length of a longest branch over all such trees with minimal diameter; that is, the length of a shortest hierarchy for $H$. Now define the {\em quasiconvex hierarchy complexity} $\mu(G)$ of a special group $G$ to be the pair of integers $(n(G),m(G))$; order the pairs $(n,m)\in\N\times\N$ lexicographically.

Now assume that all hyperbolic special groups $H$ with $\mu(H)<\mu(G)$ are good. We have a splitting of $G$ either as $A\ast_C B$ or $A\ast_C$. Note that $A,B,C$ are hyperbolic. Now $C$ is the stabiliser of a hyperplane; this is a CAT(0) cube complex whose quotient by $C$ is special and for which the intersections with other hyperplanes of $X$ give a quasiconvex hierarchy; hence $n(C)<n(G)$, so $\mu(C)<\mu(G)$ and so $C$ is good. Furthermore $A$ and $B$ have shorter hierarchies than $G$, so whether or not $n(A)=n(G)$, the complexities $\mu(A)<\mu(G)$ and $\mu(B)<\mu(G)$ do strictly decrease. Thus $A,B,C$ are good. Furthermore quasiconvex subgroups of $G$ are separable. All finite index subgroups of $A,B,C$ are quasiconvex so the splitting is efficient and we may apply Theorem \ref{goodsplitting} to conclude that $G$ is good. Note that the base case for the induction is simply the trivial group. 
\end{proof}

Recalling that Haglund and Wise \cite{HW10} proved that all Coxeter groups are virtually special, we have:
\begin{clly}\label{goodRACGs}
Hyperbolic Coxeter groups are good. 
\end{clly}

For right-angled Artin groups, Theorem \ref{goodRAAGs} guaranteed that in fact the mod-$p$ cohomology is determined by the pro-$p$ completion. This property, which is sometimes called {\em $p$-goodness}, is rather rarer than straightforward goodness; in particular proofs will often require strong separability constraints in which only $p$-group quotients are available. These constraints are difficult to obtain in general.

We move now to a result of a rather different flavour. Often, properties of the underlying graph of a right-angled Artin or Coxeter group are expressible as group theoretic properties. As an example of such a property carrying over to the profinite world, we prove the following Theorem.
\begin{theorem}\label{freeprod}
Let $\Gamma$ be a graph. Then $\widehat{A(\Gamma)}$, respectively $\widehat{C(\Gamma)}$, splits as a non-trivial profinite free product $H_1\amalg H_2$ if and only if $\Gamma$ is disconnected. 
\end{theorem}
The proof will call upon the theory of actions on profinite trees developed by, among others, Ribes and Zalesskii. The theory is contained in the unpublished book \cite{RZup}; the closely related pro-$p$ version may be found in published form in \cite{RZ00}. 
\begin{proof}[Proof of Theorem \ref{freeprod}]
If $\Gamma$ is disconnected the result follows directly from the abstract case. So suppose that $\Gamma$ is connected and that $G=\widehat{A(\Gamma)}$ splits as a profinite free product $H_1\amalg H_2$. The case when $\Gamma$ is a point is easy, so suppose that $\Gamma$ is not a point. 

The splitting of $G$ as a free profinite product induces an action of $G$ on a profinite tree $T$, where vertex stabilisers are precisely the conjugates of the $H_i$ (see Lemma 5.3.1 of \cite{RZup}). All edge stabilisers are trivial, so that no element of $G$ can fix more than one point of $T$. By Proposition 3.2.3 of \cite{RZup}, any abelian group acting on $T$ either fixes a point or is a subgroup of $\hat \Z$. Each of the standard generators of $A(\Gamma)$ is contained in a copy of $\Z{}^2$ as there is some edge adjacent to the corresponding vertex; and these copies of $\Z^2$ are retracts of the whole RAAG, hence give an inclusion of $\hat\Z{}^2$ in the profinite completion. Hence every generator of $G$ fixes some (unique) vertex of $T$, and so is contained in a (unique) conjugate of $H_1$ or $H_2$. 

Note that for every edge $e=[v,w]$ of $\Gamma$, the subgroup $\overline{<\!v,w\!>}$ is a rank 2 free abelian group so that $v,w$ fix the same vertex of $T$; by connectedness of $\Gamma$, it follows that all of $\widehat{A(\Gamma)}$ fixes a vertex of $T$ and so $H_1 = 1$ or $H_2=1$.

The case of a right-angled Coxeter group is similar; indeed it is easier, as all the generators have finite order, and therefore fix a vertex of $T$ by Theorem 3.1.7 of \cite{RZup}.   
\end{proof}

\noindent{\bf Acknowledgements} The first author was supported by the EPSRC. The second author was supported by the EPSRC and a Lamb and Flag Scholarship from St John's College Oxford. The authors would like to thank Marc Lackenby for helpful discussions regarding the paper.

\bibliographystyle{alpha}
\bibliography{RAAGs}
\end{document}